\documentclass{amsart}

\usepackage[english]{babel}

\usepackage[letterpaper,top=3cm,bottom=3cm,left=3cm,right=3cm,marginparwidth=1.75cm]{geometry}

\usepackage{amsmath}
\usepackage{graphicx}
\usepackage{tikz}
\usepackage{tikz-cd}
\usepackage{amsfonts}
\usepackage{amssymb}
\usepackage{amsthm}
\usepackage{caption}
\usepackage[colorlinks=true, allcolors=violet]{hyperref}
\usepackage{cleveref}

\newtheorem{Thm}{Theorem}[section]
\newtheorem{Cor}[Thm]{Corollary}
\newtheorem{Lm}[Thm]{Lemma}
\newtheorem{Prop}[Thm]{Proposition}

\newtheorem{LetterThm}{Theorem}

\theoremstyle{definition}
\newtheorem{Defn}[Thm]{Definition}

\newtheorem{Rem}[Thm]{Remark} 

\author{Diego Yépez}
\title[The Roots of Rank-Three Bundles on $\mathbb{P}_{\mathbb{C}}^{1}$ arising from Monodromy Representations]{Monodromy Representations: Decomposing Rank-Three Bundles over the Projective Line with Three Marked Points}
\thanks{The author is supported by the IDA Postdoctoral Fellowship}
\address{Department of Mathematics, Princeton University, Princeton, NJ 08544}
\email{dy9534@princeton.edu}
\address{IDA Center for Communications Research - Princeton, Princeton, NJ 08540}
\email{d.yepez@idaccr.org}

\begin{document}

\begin{abstract}
Given a monodromy representation $\rho$ of the projective line minus $m$ points, one can extend the resulting vector bundle with connection map canonically to a vector bundle with logarithmic connection map over all of the projective line. 
Now, since vector bundles split as twisting sheaves over the projective line, the focus of this work regards knowing the exact decomposition; i.e., computing the roots.
Particularly, we use the monodromy derivative to compute the roots for all three-dimensional $\rho$ when $m = 3$. 
\end{abstract}

\date{\today}

\maketitle

\tableofcontents

\section{Introduction}

Allow $\mathbb{P}^{1} := \mathbb{P}_{\mathbb{C}}^{1}$, and denote by $\mathbb{P}^{1}_{(p_{1}\cdots p_{m})}$ the $m$th-punctured complex projective line. Upon selecting a branch of logarithm and given a monodromy representation $\rho:\pi_{1}(\mathbb{P}^{1}_{(p_{1}\cdots p_{m})}) \rightarrow \text{GL}_{n}({\mathbb{C}})$ of $\mathbb{P}^{1}_{(p_{1}\cdots p_{m})}$, one can, following \cite{Deligne2006}, obtain a vector bundle with logarithmic connection map $(\mathcal{V}_{\text{Log}(\rho)}, \nabla_{\text{Log}(\rho)})$ over $\mathbb{P}^{1}$, which we call \textit{the associated extended logarithmic connection}. 
As is wellknown, every vector bundle over $\mathbb{P}^{1}$ decomposes into a direct sum of line bundles parametrized by $\mathbb{Z}$ \cite{G1957}. 
Naturally, one may inquire about the decomposition of the associated extended logarithmic vector bundle of a monodromy representation of $\mathbb{P}^{1}_{(p_{1}\cdots p_{m})}$. In the context of the terminology introduced in \cite{DY25}, which is derived from the terminology in the setting of vector-valued modular forms (e.g see \cite{CF2017} and \cite{FR2020}), we seek to express the decomposition as: 
\begin{center}
    \begin{eqnarray}
        \nonumber
        \mathcal{V}_{\text{Log}(\rho)} \cong \bigoplus\limits_{i = 1}^{n} \mathcal{O}(\xi_{i})
    \end{eqnarray}
\end{center}
where the $\xi_{i}$ are the \textit{roots} associated with $\rho$. 

As noted by Nitsure \cite{Lak2003}, by the results of \cite{NS1965} and \cite{MS1980}, there exists a bijective correspondence between the set of isomorphism classes of stable parabolic bundles of parabolic degree $0$ over $\mathbb{P}^{1}$ with $m$ marked points and the set of conjugacy classes of irreducible unitary monodromy representations of $\mathbb{P}^{1}_{(p_{1}\cdots p_{m})}$. 
Given the extensive literature on the moduli space of stable parabolic bundles of parabolic degree $0$ (e.g., see \cite{DP2022}, \cite{Inaba2006}, \cite{Inaba2007}, \cite{KLS2022}, \cite{TZ2008}), several articles have provided partial results relevant to our investigation, in particular, \cite{Hu2024}, \cite{Mats2024}, and \cite{MT2021}. 
The authors of \cite{Hu2024} provide the roots of irreducible unitary representations of dimension two arising from removing four or more points; in \cite{Mats2024} Matsumoto finds the roots of irreducible unitary representations of dimension three with an additional strong condition on the trace of the logarithm of the generators and arising from the removal of three points; and the writers of \cite{MT2021} provide an upper and lower bound $-m < \xi_{i} < 0$ for the roots of irreducible unitary representations of dimension $n$ with an additional condition and arising from the removal of $m$ points. 

Departing from the study of stable parabolic bundles of parabolic degree $0$, in \cite{DY25}, we introduce a systematic approach for determining the roots of any monodromy representation of $\mathbb{P}^{1}_{(p_{1}\cdots p_{m})}$ without the constraints of being unitary, irreducible, or otherwise. 
In \cite[\S $5$]{DY25} we construct, for each monodromy representation $\rho$ of $\mathbb{P}^{1}_{(p_{1}\cdots p_{m})}$, a free $\mathbb{Z}$-graded module $\mathcal{N}(\rho)$ over the ring $R := \mathbb{C}[x,y]$, where $\mathbb{P}^{1} = \text{Proj }\mathbb{C}[x,y]$. This module possesses the property:
\begin{center}
    \begin{eqnarray}
        \nonumber
        \mathcal{N}(\rho) \cong \bigoplus\limits_{i = 1}^{\textnormal{dim}(\rho)} R[\zeta_{i}]
    \end{eqnarray}
\end{center}
where $R[\zeta]$ denotes the rank-one module over $R$ obtained by shifting the grading by $\zeta \in \mathbb{Z}$. The module $\mathcal{N}(\rho)$ is termed \textit{the twisted module of $\rho$}. 
Moreover, this module is equipped with a graded derivation $D$ of degree $m-2$ when $m \geq 2$, referred to as \textit{the monodromy derivative}. The roots of $\mathcal{V}_{\text{Log}(\rho)}$ are $-\zeta_{i}$, and each $\zeta_{i}$ can be computed by combining the monodromy derivative $D$ with the first Chern class of the vector bundle. 
Using this approach, we determine the roots for any $n$-dimensional monodromy representation $\rho$ of $\mathbb{P}^{1}_{(0, \infty)}$ and for any two-dimensional monodromy representation of $\mathbb{P}^{1}_{(0, 1, \infty)}$.

The primary objective of this note is too underscore the significance of the monodromy derivative, leveraging its fundamental properties within the twisted module of $\rho$ to uncover the roots of any monodromy representation of $\mathbb{P}^{1}_{(0,1,\infty)}$ of dimension three. 
We emphasize that the approach presented herein is predominantly algebraic in nature. Furthermore, as this paper serves as a succinct continuation of \cite{DY25}, we shall reference prior results as needed; for the sake of self-containment, however, summaries of key sections will be provided where appropriate. 

Proceeding into the statements of the results, we first state a bound for the roots of any two-dimensional monodromy representation of $\mathbb{P}^{1}_{(0,1,\infty)}$: 
\begin{LetterThm}[\Cref{Root-Bound}]
    Let $\rho: \mathbb{Z}*\mathbb{Z} \rightarrow \emph{GL}_{2}(\mathbb{C})$ be a two-dimensional monodromy representation of $\mathbb{P}^{1}_{(0,1,\infty)}$ with $(\mathcal{V}_{\emph{Log}(\rho)}, \nabla_{\emph{Log}(\rho)})$ the associated extended logarithmic connection. Then for each root of $\mathcal{V}_{\emph{Log}(\rho)}$, $\xi_{i}$, where $j =1,2$ we have   
    \begin{eqnarray}
        \nonumber
        0 \geq \xi_{i} \geq -2.
    \end{eqnarray}
\end{LetterThm}

Next, we consider three-dimensional monodromy representations of $\mathbb{P}^{1}_{(0,1,\infty)}$, split them into two types, reducible and irreducible, and proceed based on type. When $\rho$ is reducible we obtain
\begin{LetterThm}[\Cref{Reducible1}] \label{IntroReducible1}
    Suppose that $0 \rightarrow \rho' \rightarrow \rho \rightarrow \chi'' \rightarrow 0$ is a short exact sequence of monodromy representations of $\mathbb{P}^{1}_{(0, 1, \infty)}$ with $\rho$ three-dimensional, $\rho'$ two-dimensional, and  $\chi''$ one-dimensional. 
    Let $0 \rightarrow \bigoplus_{i = 1}^{2}\mathcal{O}(\xi_{i}') \rightarrow \mathcal{V}_{\emph{Log}(\rho)} \rightarrow \mathcal{O}(\xi'') \rightarrow 0$ be the short exact sequence of the associated extended bundles. Arrange the roots of $\mathcal{V}_{\emph{Log}(\rho')}$ so that $\xi_{1}' \leq \xi_{2}'$. Then we have the following set of statements.
    \begin{enumerate}
        \item If $(\xi_{1}',\xi_{2}',\xi'') \not= (-2,-2,0),(-2,-1,0)$ nor $(-2,0,0)$, then the short exact sequence of associated extended bundles $0 \rightarrow \bigoplus_{i = 1}^{2}\mathcal{O}(\xi_{i}') \rightarrow \mathcal{V}_{\emph{Log}(\rho)} \rightarrow \mathcal{O}(\xi'') \rightarrow 0$ splits.
        \item If $(\xi_{1}',\xi_{2}',\xi'') = (-2,-2,0)$, then the roots of $\mathcal{V}_{\emph{Log}(\rho)}$ are either $(-2,-1,-1), \text{or }(-2,-2,0)$.
        \item If $(\xi_{1}',\xi_{2}',\xi'') = (-2,-1,0)$, then the roots of $\mathcal{V}_{\emph{Log}(\rho)}$ are either $(-2, -1, 0)$ or $(-1, -1, -1)$.
        \item If $(\xi_{1}',\xi_{2}',\xi'') = (-2,0,0)$, then the roots of $\mathcal{V}_{\emph{Log}(\rho)}$ are either $(-2,0,0)$ or $(-1,-1,0)$.
    \end{enumerate}
\end{LetterThm}

\noindent and 

\begin{LetterThm}[\Cref{Reducible2}]\label{IntroReducible2}
    Suppose that $0 \rightarrow \chi' \rightarrow \rho \rightarrow \rho'' \rightarrow 0$ is a short exact sequence of monodromy representations of $\mathbb{P}^{1}_{(0, 1, \infty)}$ with $\rho$ three-dimensional, $\chi'$ one-dimensional, and  $\rho''$ two-dimensional. 
    Let $0 \rightarrow \mathcal{O}(\xi') \rightarrow \mathcal{V}_{\emph{Log}(\rho)} \rightarrow \bigoplus_{i = 1}^{2}\mathcal{O}(\xi_{i}'') \rightarrow 0$ be the short exact sequence of the associated extended bundles. Arrange the roots of $\mathcal{V}_{\emph{Log}(\rho'')}$ so that $\xi_{1}'' \leq \xi_{2}''$. Then we have the following set of statements.
    \begin{enumerate}
        \item If $(\xi', \xi_{1}'', \xi_{2}'') \not= (-2,-2,0),(-2,-1,0)$ nor $(-2,0,0)$, then the short exact sequence of bundles $0 \rightarrow \mathcal{O}(\xi') \rightarrow \mathcal{V}_{\emph{Log}(\rho)} \rightarrow \bigoplus_{i = 1}^{2}\mathcal{O}(\xi_{i}'') \rightarrow 0$ splits.
        \item If $(\xi', \xi_{1}'', \xi_{2}'') = (-2,-2,0)$, then the roots of $\mathcal{V}_{\emph{Log}(\rho)}$ are either $(-2,-1,-1), \text{or } (-2,-2,0)$.
        \item If $(\xi', \xi_{1}'', \xi_{2}'') = (-2,-1,0)$, then the roots of $\mathcal{V}_{\emph{Log}(\rho)}$ are either $(-2, -1, 0)$ or $(-1, -1, -1)$.
        \item If $(\xi', \xi_{1}'', \xi_{2}'') = (-2,0,0)$, then the roots of $\mathcal{V}_{\emph{Log}(\rho)}$ are either $(-2,0,0)$ or $(-1,-1,0)$.
    \end{enumerate}
\end{LetterThm}

\noindent whereas when $\rho$ is irreducible 
\begin{LetterThm}[\Cref{IrreducibleReps}]\label{IntroIrreducible}
    Suppose that $\rho: \mathbb{Z}*\mathbb{Z} \rightarrow \emph{GL}_{3}(\mathbb{C})$ is an irreducible three-dimensional monodromy representation of $\mathbb{P}^{1}_{(0,1,\infty)}$ and let $(\mathcal{V}_{\emph{Log}(\rho)}, \nabla_{\emph{Log}(\rho)})$ denote the associated extended logarithmic connection. 
    Allow $\zeta$ to be the first Chern class of $\mathcal{V}_{\emph{Log}(\rho)}$.
    Then 
    \begin{eqnarray}
        \nonumber
        \mathcal{V}_{\emph{Log}(\rho)} \cong
        \begin{cases}
        \mathcal{O}(\frac{\zeta}{3})^{\oplus 3} \emph{or } \mathcal{O}(\frac{\zeta}{3} + 1) \oplus \mathcal{O}(\frac{\zeta}{3}) \oplus \mathcal{O}(\frac{\zeta}{3} - 1) & \text{when } \zeta \equiv 0 \text{ mod } 3 \\
        \mathcal{O}(\frac{\zeta + 2}{3}) \oplus \mathcal{O}(\frac{\zeta-1}{3}) \oplus \mathcal{O}(\frac{\zeta-1}{3}) & \text{when } \zeta \equiv 1 \text{ mod } 3 \\
        \mathcal{O}(\frac{\zeta+1}{3}) \oplus \mathcal{O}(\frac{\zeta+1}{3}) \oplus \mathcal{O}(\frac{\zeta-2}{3}) & \text{when } \zeta \equiv 2 \text{ mod } 3.
    \end{cases}
    \end{eqnarray}
\end{LetterThm}

\noindent Observe that \Cref{IntroReducible1}, \Cref{IntroReducible2}, and \Cref{IntroIrreducible} close the case on three-dimensional monodromy representations arising from removing three points. 
$\vspace{3mm}$

$\textbf{Conventions and Assumptions}$. 
\begin{itemize}
    \item All representations $\rho$ are assumed to be over $\mathbb{C}$.
    \item Whenever a branch of the logarithm is needed, we will always take the principal branch, even when not explicitly stated, i.e., for all $re^{2\pi i q} \in \mathbb{C}, 0 \leq q < 1$.
    \item We denote the generators of the free group on two generators with $\gamma_{0}$ and $\gamma_{1}$.
\end{itemize}
$\vspace{1mm}$

$\textbf{Acknowledgments}$. I'd like to thank Luca Candelori for, generally, bringing this question to my attention as well as for his helpful comments. The author acknowledges Louis Esser for all the lunch-time discussions around the Princeton Mathematics Department on this topic and Katy Weber for a careful reading of a previous version of this work. I'd also like to thank my postdoctoral advisor Nick Katz and the Mathematics Department at Princeton University for their hospitality.

\section{Notation \& Background}

We adhere to the notation established in \cite{DY25}, as we will extensively utilize many of the results presented therein; moreover, for the sake of self-containment, we shall briefly review the essential background while reintroducing the majority of the notation from \cite{DY25}. 

Let $\mathbb{P}^{1}_{(p_{1}, ..., p_{m})} := \mathbb{P}^{1} - \{p_{1}, ..., p_{m}\}$. 
The fundamental group of $\mathbb{P}^{1}_{(p_{1}, ..., p_{m})}$ is known to be isomorphic to the free group on $m-1$ generators, $\mathbb{Z}^{*(m-1)}$. 
Thus, a representation of the fundamental group of $\mathbb{P}^{1}_{(p_{1}, ..., p_{m})}$ is determined by the images of the generators, each of which we call a \textit{local monodromy}. 

\begin{Defn}
    A \textit{monodromy representation} of $\mathbb{P}^{1}_{(p_{1}, ..., p_{m})}$ is a representation of its fundamental group. 
\end{Defn}

\begin{Defn}
    A \textit{holomorphic connection} on $\mathbb{P}^{1}_{(p_{1}, ..., p_{m})}$ is a pair $(\mathcal{V}, \nabla)$, where $\mathcal{V}$ is a locally free sheaf on $\mathbb{P}^{1}_{(p_{1}, ..., p_{m})}$ and $\nabla: \mathcal{V} \rightarrow \mathcal{V} \otimes \Omega^{1}_{\mathbb{P}^{1}_{(p_{1}, ..., p_{m})}}$ is a morphism of sheaves satisfying the Leibniz Rule.
    When we want to refer to $\nabla$ alone we call it the \textit{connection map}.
    A morphism of connections $(\mathcal{V}, \nabla) \mapsto (\mathcal{V}', \nabla')$ is a morphism of $\mathcal{O}$-modules $\phi: \mathcal{V} \rightarrow \mathcal{V}'$ making the diagram 
    \begin{center}
    \begin{tikzcd}
    \mathcal{V} \arrow[r, "\nabla"] \arrow[d, "\phi"]
    & \mathcal{V} \otimes \Omega^{1}_{\mathbb{P}^{1}_{(p_{1}, ..., p_{m})}} \arrow[d, "\phi \otimes \text{id}"] \\
    \mathcal{V}' \arrow[r, "\nabla'"]
    & \mathcal{V} \otimes \Omega^{1}_{\mathbb{P}^{1}_{(p_{1}, ..., p_{m})}} 
    \end{tikzcd}
    \end{center}
    \noindent commute.
\end{Defn}

\begin{Thm} \label{Free-sheaf}
    Every locally free sheaf on $\mathbb{P}^{1}_{(p_{1}, ..., p_{m})}$ is free.
\end{Thm}

\begin{proof}
    See \cite[Theorem $30.4$]{Forster2012}.
\end{proof} 

\begin{Lm} \label{Equiv-cat}
    The category of finite-dimensional monodromy representations of $\mathbb{P}^{1}_{(p_{1}, ..., p_{m})}$ is equivalent to the category of holomorphic connections on $\mathbb{P}^{1}_{(p_{1}, ..., p_{m})}$.
\end{Lm}

\begin{proof}
    Combine \cite[Corollary $2.6.2$]{Szamuely2009} with \cite[Proposition $2.7.5$]{Szamuely2009}.
\end{proof}

\begin{Defn} \label{Associated-Con}
    By \Cref{Equiv-cat}, given a monodromy representation $\rho$ of $\mathbb{P}^{1}_{(p_{1}, ..., p_{m})}$, we obtain a holomorphic connection on $\mathbb{P}^{1}_{(p_{1}, ..., p_{m})}$. We call such a connection the $\textit{associated complex connection}$ and denote it by $(\mathcal{Q}_{\rho}, \nabla_{\rho})$.
\end{Defn}

\begin{Defn}
    Let $S$ be a finite number of points of $\mathbb{P}^{1}$ and $\Omega_{\mathbb{P}^{1}}^{1}(S)$ the sheaf of $1$-forms with logarithmic poles along $S$. 
    A \textit{connection with logarithmic poles along S} on $\mathbb{P}^{1}$ is a pair $(\overline{\mathcal{E}}, \overline{\nabla})$, where $\overline{\mathcal{E}}$ is a locally free sheaf on $\mathbb{P}^{1}$, and $\overline{\nabla}: \overline{\mathcal{E}} \rightarrow \overline{\mathcal{E}} \otimes_{\mathcal{O}} \Omega^{1}(S)$ is a $\mathbb{C}$-linear map satisfying the Leibniz Rule.
\end{Defn}

\begin{Lm} \label{Extension}
    Given a holomorphic connection $(\mathcal{V}, \nabla)$ on $\mathbb{P}^{1}_{(p_{1}, ..., p_{m})}$, there is a unique connection $(\overline{\mathcal{V}}, \overline{\nabla})$ on $\mathbb{P}^{1}$ with logarithmic poles along $\{p_{1}, ... ,p_{m} \}$ satisfying $(\overline{\mathcal{V}}, \overline{\nabla})|_{\mathbb{P}^{1}_{(p_{1}, ..., p_{m})}} \cong (\mathcal{V}, \nabla)$ such that the residue at each pole is the principal logarithm of the local monodromies.
    Moreover, the extension $(\overline{\mathcal{V}}, \overline{\nabla})$ of $(\mathcal{V}, \nabla)$ is functorial and exact in $(\mathcal{V}, \nabla)$.
\end{Lm}

\begin{proof}
    See \cite[Proposition $5.4$]{Deligne2006}.
\end{proof}

\begin{Defn} \label{Extended-Log-Definition}
    As a result of applying \Cref{Extension} to \Cref{Equiv-cat} we can associate to each monodromy representation $\rho$ of $\mathbb{P}^{1}_{(p_{1}, ..., p_{m})}$ a unique logarithmic connection over $\mathbb{P}^{1}$. 
    We refer to this connection as the $\textit{associated extended logarithmic connection}$ and denote it as $(\mathcal{V}_{\text{Log}(\rho)}, \nabla_{\text{Log}(\rho)})$.
\end{Defn}

\begin{Defn} \label{Roots}
    In the setting of \Cref{Extended-Log-Definition}, as $\mathcal{V}_{\text{Log}(\rho)} \cong \bigoplus_{j = 1}^{\text{dim} (\rho)} \mathcal{O}(\xi_{j})$ we call the twisting parameters $\xi_{j}$ the \textit{roots}.
\end{Defn}

\begin{Lm} \label{First-Chern-Class}
    The first Chern class of $\mathcal{V} \in \textnormal{Vect}(\mathbb{P}^{1})$, denoted $c_{1}(\mathcal{V})$, is the sum of the twisting parameters. 
\end{Lm}

\begin{proof}
    See \cite[Lemma $2.10$]{DY25}.
\end{proof}

\begin{Thm}
    Let $(\mathcal{V}, \nabla)$ be a connection on $\mathbb{P}^{1}$ with logarithmic poles along $\{p_{1}, ..., p_{m}\}$, then the following holds:
    \begin{eqnarray}
    \label{Ohtsuki}
    \text{c}_{1}(\mathcal{V}) = - \sum\limits_{j=1}^{m} \emph{Tr}(\emph{Res}_{p_{j}}\nabla)
\end{eqnarray}

\noindent where $\emph{Res}_{p_{j}}\nabla$ is the residue of $\nabla$ at the pole $p_{j}$. 
\end{Thm}

\begin{proof}
    See \cite{Ohtsuki1982}.
\end{proof}

\begin{Rem}
    From \Cref{Ohtsuki} one can explicitly compute the first Chern class of any associated extended logarithmic connection from the monodromy representation. 
    Additionally, note that any higher Chern class vanishes for dimension reasons.
\end{Rem}

\begin{Defn}
   In the language of stable parabolic vector bundles, Equation $(\ref{Ohtsuki})$ goes by the Fuchs relation, but here we refer to it as $\textit{Ohtsuki's formula}$ to maintain consistency with \cite[\S $2$]{DY25}.  
\end{Defn} 

\begin{Rem}
    \Cref{Free-sheaf}, \Cref{Extension}, and \Cref{Ohtsuki} are special cases of the much more general statements proved in their respective sources.
\end{Rem}

\section{Summary of the Monodromy Derivative}\label{Summary}

When the representation $\rho$ is reducible, we take as a point of departure the short exact sequence $0 \rightarrow \rho' \rightarrow \rho \rightarrow \rho'' \rightarrow 0$, where a priori $ 0 < \text{dim}(\rho'), \text{dim}(\rho'') < \text{dim}(\rho)$. 
By combining \Cref{Equiv-cat} with \Cref{Extension} the short exact sequence of monodromy representations of $\mathbb{P}^{1}_{(0,1,\infty)}$ produces a short exact sequence $0 \rightarrow \mathcal{V}_{\text{Log}(\rho')} \rightarrow \mathcal{V}_{\text{Log}(\rho)} \rightarrow \mathcal{V}_{\text{Log}(\rho'')} \rightarrow 0$ of associated extended logarithmic bundles on $\mathbb{P}^{1}$. 
From here we find a lower and upper bound for the roots of $\mathcal{V}_{\text{Log}(\rho')}, \mathcal{V}_{\text{Log}(\rho'')}$ and combine the bounds with the long exact sequence in sheaf cohomology to find the roots of $\mathcal{V}_{\text{Log}(\rho)}$ (see \Cref{ReducibleSection}).

When the representation is irreducible, the approach above is futile, because our use of the long exact sequence in sheaf cohomology heavily relies on the fact that the short exact sequence of associated extended vector bundles corresponds to a short exact sequence of monodromy representations. 
Since $\rho$ is irreducible, we cannot construct such a short exact sequence in this scenario. This leaves us in search of a new apparatus to apply to this situation.
After observing the use of the modular derivative in papers such as \cite{MM2010}, \cite{Mason2008} and particularly in \cite{CHMY2018}, in \cite{DY25} we set out to construct an analogue of the modular derivative. 

For completeness and an attempt at self-containment, we provide a brief account of such constructions and state some useful facts. We summarize the construction below, concluding with \Cref{monodromy-derivative}.

\subsection{Auxiliary Connection}
The purpose of this subsection is to develop an auxiliary connection that acts on $\mathcal{O}(1)$. 

Take the character representation $\chi: \pi_{1}(\mathbb{P}^{1}_{(0,1,\infty)}) \rightarrow \mathbb{C}^{*}$ defined by mapping both generators to $-1$; applying \cite[Proposition $4.1$]{DY25}, the resulting associated extended logarithmic connection is $(\mathcal{V}_{\text{Log}(\chi)}, \nabla_{\text{Log}(\chi)})$ over $\mathbb{P}^{1}$ with $\mathcal{V}_{\text{Log}(\chi)} \cong \mathcal{O}(-1)$.
Furthermore, after taking the dual, we have $(\mathcal{O}(1), \nabla_{\text{Log}(\chi)}^{*})$ where we drop the subscript and the superscript and refer to the connection map simply as $\nabla$.
As $\nabla$ is a connection map it follows that
\begin{eqnarray}
    \nabla: \mathcal{O}(1) \rightarrow \mathcal{O}(1) \otimes_{\mathcal{O}} \Omega_{\mathbb{P}^{1}}^{1}([0]+[1]+[\infty])
\end{eqnarray}
\noindent where  
\begin{eqnarray}
    \nonumber
    \Omega_{\mathbb{P}^{1}}^{1}([0]+[1]+[\infty]) \cong \mathcal{O}(1),
\end{eqnarray}

\noindent consequently,
\begin{eqnarray}
    \nabla: \mathcal{O}(1) \rightarrow \mathcal{O}(2).
\end{eqnarray}

At this point we tensor $\nabla$ with itself in such a way that the resulting object is a connection map. Described explicitly on sections, for all sections $v,t \in O(1)$ the following is satisfied: $(\nabla \otimes \nabla)(v \otimes t) = \nabla(v) \otimes t + v \otimes \nabla(t)$. 
Therefore,  
\begin{eqnarray}
    \nabla^{\otimes 2}: \mathcal{O}(2) \rightarrow \mathcal{O}(3)
\end{eqnarray}
and so, inductively,  we define $\nabla_{\xi} := \nabla^{\otimes \xi}$. Now, note that 
\begin{eqnarray}
    \nabla_{\xi}: \mathcal{O}(\xi) \rightarrow \mathcal{O}(\xi+1).
\end{eqnarray}

\begin{Defn} \label{little-delta}
    We refer to $\nabla_{\xi}$ as the \textit{auxiliary connection map of weight} $\xi$ \textit{arising from} $\mathbb{P}^{1}_{(0,1,\infty)}$.
\end{Defn}

\subsection{Monodromy Derivative} 
Giving immediate validation to the naming of the auxiliary connection, when we consider a monodromy representation $\rho$ of $\mathbb{P}^{1}_{(0,1,\infty)}$, we view the associated extended logarithmic connection as a triple $(\mathcal{V}_{\text{Log}(\rho)}, \nabla_{\text{Log}(\rho)}, \nabla)$, where $\nabla$ refers to the auxiliary connection. 

Witness that we now have the new gizmo
\begin{eqnarray}
    \nabla_{Log(\rho)} \otimes \nabla: \mathcal{V}_{Log(\rho)} \otimes_{\mathcal{O}} \mathcal{O}(1) \rightarrow \mathcal{V}_{Log(\rho)} \otimes_{\mathcal{O}} \mathcal{O}(2) 
\end{eqnarray}

\noindent and more generally, considering an auxiliary connection map of arbitrary weight 
\begin{eqnarray}
    \nabla_{Log(\rho)} \otimes \nabla_{\xi}: \mathcal{V}_{Log(\rho)} \otimes_{\mathcal{O}} \mathcal{O}(\xi) \rightarrow \mathcal{V}_{Log(\rho)} \otimes_{\mathcal{O}} \mathcal{O}(\xi+1). 
\end{eqnarray}

\begin{Defn} \label{twisted-module}
    Turning our attention to global sections, define $\mathcal{N}_{\xi}(\rho) := H^{0}(\mathbb{P}^{1}, \mathcal{V}_{\text{Log}(\rho)} \otimes_{\mathcal{O}} \mathcal{O}(\xi))$, and furthermore, allow $\mathcal{N}(\rho) := \bigoplus_{\xi \in \mathbb{Z}} \mathcal{N}_{\xi}(\rho)$.
    The direct sum $\mathcal{N}(\rho)$ goes by \textit{the twisted module of $\rho$}.
\end{Defn}

\begin{Lm}\label{ModuleProperties}
    Given a monodromy representation $\rho$ of $\mathbb{P}^{1}_{(0,1,\infty)}$, $\mathcal{N}(\rho)$ has the following properties. 
    \begin{enumerate}
        \item There exists an integer $\xi_{\text{min}}$ such that for any $\xi_{j} < \xi_{\text{min}}$ we have  $\mathcal{N}_{\xi_{j}}(\rho) = 0$, which is referred to as \emph{the minimal weight}.
        \item Taking $\mathbb{P}^{1} = \emph{Proj }\mathbb{C}[x,y]$, let $R:= \mathbb{C}[x,y]$. Now, $\mathcal{N}(\rho)$ is a $\mathbb{Z}$-graded module of global sections over the ring $\mathcal{N}(\textbf{1}) \cong R$, where $\textbf{1}$ is the trivial one-dimensional representation. Additionally, there is an isomorphism $\mathcal{N}(\rho) \cong \bigoplus_{i = 1}^{\emph{dim}(\rho)} R[-\xi_{i}]$, where by $R[b]$ we denote the rank-one graded module over $R$ obtained by shifting the grading by $b$. The $-\xi_{i}$ are called the \emph{generating weights}. 
        \item The roots of $\mathcal{V}_{\emph{Log}(\rho)}$ are the negatives of the generating weights for $\mathcal{N}(\rho)$.
    \end{enumerate}
\end{Lm}

\begin{proof}
    See \cite[\S $5$]{DY25}.
\end{proof}

On global sections, from the new connection map $\nabla_{\emph{Log}(\rho)} \otimes \nabla_{\xi}$ we obtain a graded derivation of degree one. 
Indeed, let us denote the connection map $\nabla_{\emph{Log}(\rho)} \otimes \nabla_{\xi}$ as $D_{\xi}$ when applied to global sections, so that
\begin{eqnarray}
    D_{\xi}: \mathcal{N}_{\xi}(\rho) \rightarrow \mathcal{N}_{\xi+1}(\rho)
\end{eqnarray}

\noindent where of course $D_{\xi}$ satisfies the Leibniz rule, as $\nabla_{Log(\rho)} \otimes \nabla_{\xi}$ is a connection map, and so 
\begin{eqnarray}
    D := \bigoplus_{\xi \in \mathbb{Z}} D_{\xi}
\end{eqnarray}

\noindent is a graded derivation of degree one on $\mathcal{N}(\rho)$.

\begin{Rem}\label{Degree-of-D}
    The degree of the derivation is a consequence of the number of points removed from $\mathbb{P}^{1}$. The degree of the derivation will be $m-2$ for $m>1$, where $m \in \mathbb{N}$ corresponds to the number of points removed. This is essentially a consequence of the isomorphism 
    \begin{eqnarray}
        \nonumber
        \Omega^{1}_{\mathbb{P}^{1}} \Big(\sum_{i=1}^{m}[p_{i}] \Big) \cong \mathcal{O}(m-2).
    \end{eqnarray}
    where of course the $p_{i}$ represent the points removed. 
    Consult with \cite[Remark $5.6$]{DY25} for $\mathcal{N}(\rho)$ when $\rho$  is a monodromy representation of $\mathbb{P}^{1}_{(p_{1}, ..., p_{m})}$. 
\end{Rem}

\begin{Defn} \label{monodromy-derivative}
    We refer to the graded derivation $D$ as the \textit{monodromy derivative of} $\rho$.
\end{Defn}

To culminate this section we mention that the monodromy derivative and its properties will be heavily utilized in \Cref{IrreducibleSection}.

\section{Reducible Monodromy Representations} \label{ReducibleSection}
The goal of this section is to compute the roots of any vector bundle arising from a reducible three-dimensional monodromy representation of $\mathbb{P}^{1}_{(0, 1, \infty)}$. 
As mentioned in the previous section, the strategy is to first find a lower and an upper bound for the roots of $\mathcal{V}_{\text{Log}(\rho)}$ when $\rho$ is a two-dimensional monodromy representation of $\mathbb{P}^{1}_{(0, 1, \infty)}$, and combine those bounds with sheaf cohomology.

\begin{Lm} \label{Chern-Bound}
    Let $\rho: \mathbb{Z}*\mathbb{Z} \rightarrow \emph{GL}_{2}(\mathbb{C})$ be a two-dimensional monodromy representation of $\mathbb{P}^{1}_{(0,1,\infty)}$ with $(\mathcal{V}_{\emph{Log}(\rho)}, \nabla_{\emph{Log}(\rho)})$ the associated extended logarithmic connection. Then 
    \begin{eqnarray}
        \nonumber
        0 \geq c_{1}(\mathcal{V}_{\emph{Log}(\rho)}) \geq -4.
    \end{eqnarray}
\end{Lm}

\begin{proof}
    Let $\rho: \mathbb{Z}*\mathbb{Z} \rightarrow \text{GL}_{2}(\mathbb{C})$ be a two-dimensional representation defined by $\gamma_{0} \mapsto M_{0}$ and $\gamma_{1} \mapsto M_{1}$, where $\lambda_{0}, \mu_{0}$ and $\lambda_{1}, \mu_{1}$ are the eigenvalues of $M_{0}$ and $M_{1}$, respectively. Further, let $\lambda_{0} = r_{\lambda,0}e^{2\pi i q_{\lambda,0}}, \mu_{0} = r_{\mu,0}e^{2\pi i q_{\mu,0}}, \lambda_{1} = r_{\lambda, 1} e^{2\pi i q_{\lambda,1}},$ and $\mu_{1} = r_{\mu,1}e^{2\pi i q_{\mu,1}}$. 
    Similarly, we allow $\nu_{0} = r_{\nu,0}e^{2\pi i q_{\nu,0}}$ and $\nu_{1} = r_{\nu,1}e^{2\pi i q_{\nu,1}}$ to be the eigenvalues of $M_{0}M_{1}$.
    We remark that we will explicitly establish the bounds when $\rho$ is irreducible, as the reducible case will follow the exact same procedure. 
    For now, suppose that $\rho$ is irreducible. By \cite[Ch $2$, Theorem $4.2.1$]{IKSY2013} there exists a basis $\hat{\Upsilon}$ in which
    \begin{eqnarray}
        \nonumber
        M_{0} \leftrightarrow 
        \begin{pmatrix}
            \lambda_{0} && 1 \\
            0 && \mu_{0} 
        \end{pmatrix} \text{ and  }
        M_{1} \leftrightarrow 
        \begin{pmatrix}
            \lambda_{1} && 0 \\
            \sigma && \mu_{1} 
        \end{pmatrix}
    \end{eqnarray}

    \noindent where $\sigma = (\nu_{0} + \nu_{1}) - (\lambda_{0}\lambda_{1} + \mu_{0}\mu_{1}) \not= 0$. 
    We will employ Ohtsuki's formula to bound $c_{1}(\mathcal{V}_{\text{Log}(\rho)})$. Looking at the trace of the residues of the connection map, it is clear that $\text{Tr}(\text{log}(M_{0})) = \text{log}(\lambda_{0}) + \text{log}(\mu_{0})$ and $\text{Tr}(\text{log}(M_{1})) = \text{log}(\lambda_{1}) + \text{log}(\mu_{1})$.
    Recall the fact that if $S \in \text{GL}_{n}(\mathbb{C})$ and $T \in \text{Mat}(n \times n, \mathbb{C})$ then
    \begin{eqnarray} \label{Jf}
        \text{exp}(S^{-1}TS) = S^{-1}\text{exp}(T)S.
    \end{eqnarray}
    Observe that as a consequence, to find $\text{Tr}(\text{log}((M_{0}M_{1})^{-1}))$, we may suppose that $M_{0}M_{1}$ is in Jordan form. Then it is clear that $\text{Tr}(\text{log}((M_{0}M_{1})^{-1})) = \text{log}(\nu_{0}^{-1}) + \text{log}(\nu_{1}^{-1})$. Using Ohtsuki's formula, we have 
    \begin{eqnarray}
        \nonumber
        c_{1}(\mathcal{V}_{\text{Log}(\rho)}) = -\text{log}(\lambda_{0}) - \text{log}(\mu_{0}) - \text{log}(\lambda_{1}) - \text{log}(\mu_{1}) - \text{log}(\nu_{0}^{-1}) - \text{log}(\nu_{1}^{-1}),  
    \end{eqnarray}
    which, after rearranging the real and imaginary components, we write
    \begin{eqnarray}
        \nonumber
        c_{1}(\mathcal{V}_{\text{Log}(\rho)}) = -\text{log}(r_{\gamma,0}r_{\mu,0}r_{\gamma,1}r_{\mu,1}r_{\nu,0}^{-1}r_{\nu,1}^{-1}) - q_{\lambda,0} - q_{\mu,0} - q_{\lambda,1} - q_{\mu,1} - (-q_{\nu,0}+a - q_{\nu,1} + b),  
    \end{eqnarray}
    such that $a,b \in \mathbb{Z}_{\geq 0}$ appear from the fact that the exponent of each $\nu_{j}^{-1}$ with $j=1,2$ must be an element of $[0,1)$ by the choice of logarithm. 
    Since the determinant of the product is the product of the determinants, it follows that $r_{\gamma,0}r_{\mu,0}r_{\gamma,1}r_{\mu,1}r_{\nu,0}^{-1}r_{\nu,1}^{-1} = 1$ and $q_{\lambda,0} + q_{\mu,0} + q_{\lambda,1} + q_{\mu,1} = q_{\nu,0} + q_{\nu,1}$. Thus, 
    \begin{eqnarray}
        c_{1}(\mathcal{V}_{\text{Log}(\rho)}) = -a -b.
    \end{eqnarray}
    Now, when all the $q$s are zero, then $a = b = 0$, establishing the upper bound. 
    For the lower bound, since each $q_{\lambda,i} \in [0,1)$, and $q_{\mu,j} \in [0,1)$ for $i,j \in \{0,1\}$, $q_{\lambda,0} + q_{\mu,0} + q_{\lambda,1} + q_{\mu,1} < 4$, and so when $3 < q_{\lambda,0} + q_{\mu,0} + q_{\lambda,1} + q_{\mu,1} < 4$, we must have that $a+ b = 4$. 

    For the case when $\rho$ is reducible, first notice that as a consequence of \cite[Ch $2$, Theorem $4.2.1$]{IKSY2013} we may assume that
    \begin{eqnarray}
        \nonumber
        M_{0} \leftrightarrow 
        \begin{pmatrix}
            \lambda_{0} && \vartheta_{1} \\
            0 && \mu_{0} 
        \end{pmatrix} \text{ and  }
        M_{1} \leftrightarrow 
        \begin{pmatrix}
            \lambda_{1} && \vartheta_{2} \\
            0 && \mu_{1} 
        \end{pmatrix}
    \end{eqnarray}
    with $\vartheta_{1},\vartheta_{2} \in \mathbb{C}$. Employing Equation $(\ref{Jf})$ to find $\text{Tr}(\text{log}((M_{0}M_{1})^{-1}))$, we may suppose that $M_{0}M_{1}$ is in Jordan form. By Ohtsuki's formula we then obtain 
    \begin{eqnarray}
        \nonumber
        c_{1}(\mathcal{V}_{\text{Log}(\rho)}) = -\text{log}(\lambda_{0}) - \text{log}(\mu_{0}) - \text{log}(\lambda_{1}) - \text{log}(\mu_{1}) - \text{log}(\nu_{0}^{-1}) - \text{log}(\nu_{1}^{-1}),  
    \end{eqnarray}
    and the remainder of the proof is identical to the case when $\rho$ is irreducible. 

\end{proof}

\begin{Cor} \label{Parabolic_Consistent}
    Let $\rho: \mathbb{Z}*\mathbb{Z} \rightarrow \emph{GL}_{2}(\mathbb{C})$ be an irreducible unitary two-dimensional monodromy representation of $\mathbb{P}^{1}_{(0,1,\infty)}$ with $(\mathcal{V}_{\emph{Log}(\rho)}, \nabla_{\emph{Log}(\rho)})$ the associated extended logarithmic connection. Then
    \begin{eqnarray}
        \nonumber
        0 > c_{1}(\mathcal{V}_{\emph{Log}(\rho)}) \geq -4.
    \end{eqnarray}
\end{Cor}

\begin{proof}
    As $\rho$ is irreducible, by \cite[Ch $2$, Theorem $4.2.1$]{IKSY2013}, the product of any eigenvalue of $\rho(\gamma_{0})$ with any eigenvalue of $\rho(\gamma_{1})$ does not equal an eigenvalue of $\rho(\gamma_{0}\gamma_{1})$. Considering the part of the proof of \Cref{Chern-Bound} when $\rho$ is irreducible, observe that since $\rho$ is unitary, we must have $r_{i,j} = 1$ for all $i \in \{\gamma, \mu, \nu\}, j \in \{0,1\}$.
    In order to obtain $c_{1}(\mathcal{V}_{\text{Log}(\rho)}) = 0$, we need $a = b = 0$, but this only occurs when $q_{\lambda,i} = 0, q_{\mu,j} = 0$, and $q_{\nu,k} = 0$ for $i,j,k \in \{0,1\}$. Witness then that, if $c_{1}(\mathcal{V}_{\text{Log}(\rho)}) = 0$, then $q_{\lambda,i} = 0, q_{\mu,j} = 0$, and $q_{\nu,k} = 0$ for $i,j,k \in \{0,1\}$ along with $r_{i,j} = 1$ for all $i \in \{\gamma, \mu, \nu\}, j \in \{0,1\}$, contradicting the fact that the product of any eigenvalue of $\rho(\gamma_{0})$ with any eigenvalue of $\rho(\gamma_{1})$ does not equal an eigenvalue of $\rho(\gamma_{0}\gamma_{1})$. Whence, the result $c_{1}(\mathcal{V}_{\text{Log}(\rho)}) \not= 0$.
\end{proof}

\begin{Lm} \label{Root-Bound}
    Let $\rho: \mathbb{Z}*\mathbb{Z} \rightarrow \emph{GL}_{2}(\mathbb{C})$ be a two-dimensional monodromy representation of $\mathbb{P}^{1}_{(0,1,\infty)}$ with $(\mathcal{V}_{\emph{Log}(\rho)}, \nabla_{\emph{Log}(\rho)})$ the associated extended logarithmic connection. Then for each root of $\mathcal{V}_{\emph{Log}(\rho)}$, $\xi_{i}$, where $j =1,2$ we have   
    \begin{eqnarray}
        \nonumber
        0 \geq \xi_{i} \geq -2.
    \end{eqnarray}
\end{Lm}

\begin{proof}
    With regard to the case when $\rho$ is irreducible, combine \cite[Theorem $5.9$]{DY25} with \Cref{First-Chern-Class} and \Cref{Chern-Bound}. Regarding the case when $\rho$ is reducible, combine \cite[Theorem $4.4$]{DY25} with \cite[Proposition $4.1$]{DY25}.
\end{proof}

\begin{Rem}
    \Cref{Parabolic_Consistent} has been known by the individuals who work on the moduli space of stable parabolic vector bundles with logarithmic connection over $\mathbb{P}_{\mathbb{C}}^{1}$. 
    The special case of \Cref{Root-Bound} for which $\rho$ is irreducible and unitary is also known by the same set of individuals; for a geometric proof of the bound when $\rho$ is irreducible \& unitary see, for example, \cite[\S $2$]{MT2021}.
    
    The proof of the following proposition along with the proof of \cite[Lemma $5.8$]{DY25} and the proof of \Cref{Simple-Module}, highlight the standard trick regarding $\mathcal{N}(\rho)$ and $D$.
\end{Rem}

\begin{Prop}\label{NonRoots}
    Let $\rho$ be a three-dimensional reducible monodromy representation of $\mathbb{P}^{1}_{(0,1,\infty)}$. Then $\mathcal{V}_{\emph{Log}(\rho)} \not\cong \mathcal{O}(-3) \oplus \mathcal{O}(-1) \oplus \mathcal{O}(0)$.
\end{Prop}

\begin{proof}
    Assume that $\mathcal{V}_{\text{Log}(\rho)} \cong \mathcal{O}(-3) \oplus \mathcal{O}(-1) \oplus \mathcal{O}(0)$. Following \Cref{Summary}, we can construct $\mathcal{N}(\rho)$ and, furthermore, by \Cref{ModuleProperties}, we know that the negatives of the roots are the generating weights of $\mathcal{N}(\rho)$. Now let $F, G$, and $H$ generate $\mathcal{N}(\rho)$ with weights $0,1,3$ respectively. 

    By construction, $D_{0}(F) \in \mathcal{N}_{1}(\rho)$, and so $D_{0}(F)$ has weight $1$. 
    Moreover, since $F,G$, and $H$ generate $\mathcal{N}(\rho)$, we must have $D_{0}(F) = \eta_{F} F + \tau_{F} G + \omega_{F} H$ for some $\eta_{F}, \tau_{F}, \omega_{F} \in \mathbb{C}[x,y]$. Since $H \in \mathcal{N}_{3}(\rho)$, then in fact $\omega_{F} = 0$. 
    Similarly, by construction, $D_{1}(G) \in \mathcal{N}_{2}(\rho)$, where $D_{1}(G) = \eta_{G} F + \tau_{G} G + \omega_{G} H$ for $\eta_{F}, \tau_{F}, \omega_{F} \in \mathbb{C}[x,y]$; but, as $H \in \mathcal{N}_{3}(\rho)$, then in fact $\omega_{G} = 0$. 

    It follows that $F$ and $G$ generate a submodule $L$ of $\mathcal{N}(\rho)$ that is closed under $D$. 
    Observe that $F|_{\mathbb{P}^{1}_{(0,1,\infty)}}$ and $G|_{\mathbb{P}^{1}_{(0,1,\infty)}}$ are nonzero global section of $\mathcal{Q}_{\rho}$ (see \Cref{Associated-Con}), since $\big(\mathcal{V}_{\text{Log}(\rho)} \otimes \mathcal{O}(\xi)\big)|_{\mathbb{P}^{1}_{(0,1,\infty)}} \cong \mathcal{Q}_{\rho}$.
    Denote with $\mathcal{E}_{L} \subset \mathcal{Q}_{\rho}$ the subsheaf generated by $L$ on $\mathbb{P}^{1}_{(0,1,\infty)}$.
    Now, denoting with $\varphi$ the restriction maps, as the following diagram 
    \begin{center}
    \begin{tikzcd}
    \mathcal{V}_{\text{Log}(\rho)} \otimes \mathcal{O}(\xi) \arrow[r, "\nabla_{\text{Log}(\rho)} \otimes \nabla"] \arrow[d, "\varphi"]
    & \mathcal{V}_{\text{Log}(\rho)} \otimes \mathcal{O}(\xi+1) \arrow[d, "\varphi"] \\
    \mathcal{Q}_{\rho} \arrow[r, "(\nabla_{\text{Log}(\rho)} \otimes \nabla)|_{\mathbb{P}^{1}_{(0,1,\infty)}}"]
    & \mathcal{Q}_{\rho}  
    \end{tikzcd}
    \end{center}
    commutes, then, $(DF)|_{\mathbb{P}^{1}_{(0,1,\infty)}} = D|_{\mathbb{P}^{1}_{(0,1,\infty)}}(F|_{\mathbb{P}^{1}_{(0,1,\infty)}})$ and $(DG)|_{\mathbb{P}^{1}_{(0,1,\infty)}} = D|_{\mathbb{P}^{1}_{(0,1,\infty)}}(G|_{\mathbb{P}^{1}_{(0,1,\infty)}})$.
    As every vector bundle over the affine space $\mathbb{P}^{1}_{(0,1,\infty)} = \text{Spec}(A)$ is free by \Cref{Free-sheaf}, $\mathcal{E}_{L}$ can be viewed as a free module of rank two over $A$ whose generators are stable under $D|_{\mathbb{P}^{1}_{(0,1,\infty)}} \cong \nabla_{\rho}$, then $\mathcal{E}_{L}$ inherits the connection map $\nabla_{\rho}|_{\mathcal{E}_{L}}$ in such a way where $(\mathcal{E}_{L},\nabla_{\rho}|_{\mathcal{E}_{L}}) \subset (\mathcal{Q}_{\rho}, \nabla_{\rho})$ as connections.

    Witness that since there is an equivalence of categories between holomorphic connections on $\mathbb{P}^{1}_{(0,1,\infty)}$ and monodromy representations of $\mathbb{P}^{1}_{(0,1,\infty)}$, then $\rho' \subset \rho$, where $\rho'$ corresponds to $(\mathcal{E}_{L}, \nabla_{\rho}|_{\mathcal{E}_{L}})$. 
    Examining $(\mathcal{E}_{L}, \nabla_{\rho}|_{\mathcal{E}_{L}})$, one sees that $(\mathcal{E}_{L}, \nabla_{\rho}|_{\mathcal{E}_{L}})$ is nonempty, which means that $\rho' \not= 0$; furthermore, as $\mathcal{N}(\rho)$ is of rank three, it follows that $(\mathcal{E}_{L}, \nabla_{\rho}|_{\mathcal{E}_{L}}) \subsetneqq (\mathcal{Q}_{\rho}, \nabla_{\rho})$, which forces $\rho' \subsetneqq \rho$.

    At this point, we have $0 \rightarrow \rho' \rightarrow \rho \rightarrow \chi \rightarrow 0$, which produces $0 \rightarrow \mathcal{V}_{\text{Log}(\rho')} \rightarrow \mathcal{V}_{\text{Log}(\rho)} \rightarrow \mathcal{V}_{\text{Log}(\chi)} \rightarrow 0$; as $\text{c}_{1}(\mathcal{V}_{\text{Log}(\rho')}) = -1$ and $\text{c}_{1}(\mathcal{V}_{\text{Log}(\rho)}) = -4$, then $\mathcal{V}_{\text{Log}(\chi)} \cong \mathcal{O}(-3)$. However, $\mathcal{V}_{\text{Log}(\chi)}$ is of rank one and degree $-3$, so it contradicts \cite[Proposition $4.1$]{DY25}.
\end{proof}

Now that we have collected all the required preliminary statements, we are in a position to decompose the associated extended bundles of any reducible three-dimensional monodromy representation of $\mathbb{P}^{1}_{(0,1,\infty)}$. We begin by considering such representations that admit a two-dimensional subrepresentation. 

\begin{Thm} \label{Reducible1}
    Suppose that $0 \rightarrow \rho' \rightarrow \rho \rightarrow \chi'' \rightarrow 0$ is a short exact sequence of monodromy representations of $\mathbb{P}^{1}_{(0, 1, \infty)}$ with $\rho$ three-dimensional, $\rho'$ two-dimensional, and  $\chi''$ one-dimensional. 
    With the principal branch as a choice of logarithm, let $0 \rightarrow \bigoplus_{i = 1}^{2}\mathcal{O}(\xi_{i}') \rightarrow \mathcal{V}_{\emph{Log}(\rho)} \rightarrow \mathcal{O}(\xi'') \rightarrow 0$ be the short exact sequence of the associated extended vector bundles. Arrange the roots of $\mathcal{V}_{\emph{Log}(\rho')}$ so that $\xi_{1}' \leq \xi_{2}'$. Then we have the following set of statements.
    \begin{enumerate}
        \item If $(\xi_{1}',\xi_{2}',\xi'') \not= (-2,-2,0),(-2,-1,0)$, nor $(-2,0,0)$, then the short exact sequence of associated extended bundles $0 \rightarrow \bigoplus_{i = 1}^{2}\mathcal{O}(\xi_{i}') \rightarrow \mathcal{V}_{Log(\rho)} \rightarrow \mathcal{O}(\xi'') \rightarrow 0$ splits.
        \item If $(\xi_{1}',\xi_{2}',\xi'') = (-2,-2,0)$, then the roots of $\mathcal{V}_{\emph{Log}(\rho)}$ are either $(-2,-1,-1) \text{or }(-2,-2,0)$.
        \item If $(\xi_{1}',\xi_{2}',\xi'') = (-2,-1,0)$, then the roots of $\mathcal{V}_{\emph{Log}(\rho)}$ are either $(-2, -1, 0)$ or $(-1, -1, -1)$.
        \item If $(\xi_{1}',\xi_{2}',\xi'') = (-2,0,0)$, then the roots of $\mathcal{V}_{\emph{Log}(\rho)}$ are either $(-2,0,0)$ or $(-1,-1,0)$.
    \end{enumerate}
\end{Thm}

\begin{proof}
    Using cohomology, 
    \begin{eqnarray}
    \nonumber H^{1}(\mathbb{P}^{1}, (\mathcal{O}(\xi'')^{*} \otimes \bigoplus_{i=1}^{2}\mathcal{O}(\xi_{i}')) && = H^{1}(\mathbb{P}^{1}, \mathcal{O}(-\xi'') \otimes \bigoplus_{i=1}^{2}\mathcal{O}(\xi_{i}')) \\
    \nonumber &&= H^{1}(\mathbb{P}^{1}, \bigoplus_{i=1}^{2}\mathcal{O}(-\xi''+\xi_{i}')).  
    \end{eqnarray}

    \noindent Further, by Serre-duality, we have
    \begin{eqnarray}
    \nonumber H^{1}(\mathbb{P}^{1}, \bigoplus_{i=1}^{2}\mathcal{O}(-\xi''+\xi_{i}')) &&= H^{0}(\mathbb{P}^{1}, (\bigoplus_{i=1}^{2}\mathcal{O}(-\xi''+\xi_{i}')^{*} \otimes \Omega_{\mathbb{P}^{1}}^{1}))^{*} \\
    \nonumber &&= H^{0}(\mathbb{P}^{1}, \bigoplus_{i=1}^{2}\mathcal{O}(\xi''-\xi_{i}'-2))^{*}.
    \end{eqnarray}

    \noindent Hence, if 
    \begin{eqnarray} \label{Inequalities}
        \xi'' - \xi_{1}' < 2 \text{ and } \xi'' - \xi_{2}' < 2,
    \end{eqnarray}
    then $0 \rightarrow \bigoplus_{i = 1}^{2}\mathcal{O}(\xi_{i}') \rightarrow \mathcal{V}_{\text{Log}(\rho)} \rightarrow \mathcal{O}(\xi'') \rightarrow 0$ splits.

    For $(1)$, by \cite[Proposition $4.1$]{DY25}, we know that $0 \geq \xi'' \geq -2$, and, moreover, by \Cref{Root-Bound}, we have $0 \geq \xi_{i}' \geq -2$ for $i=1,2$. Consequently, the inequalities (\ref{Inequalities}) are not satisfied only when $\xi'' = 0$ and at least one of the $\xi_{i}' = -2$. However, without loss of generality, if $\xi_{1}' = -2$, then as a consequence of \cite[Theorem $5.9$]{DY25}, $\xi_{2}' = -2$ or $-1$ for the case when $\rho'$ is irreducible; whereas when $\rho'$ is reducible, as consequence of \cite[Theorem $4.4$]{DY25} we also see that $\xi_{2}'$ can be $0$. This establishes $(1)$. 

    Regarding the remaining three cases, we resort to studying the long exact sequence in cohomology. We have 
    \begin{eqnarray}
    \nonumber
    && 0 \rightarrow H^{0}(\mathbb{P}^{1}, \bigoplus_{i=1}^{2}\mathcal{O}(\xi_{i}')) \rightarrow H^{0}(\mathbb{P}^{1}, \mathcal{V}_{\text{Log}(\rho)}) \rightarrow H^{0}(\mathbb{P}^{1}, \mathcal{O}(\xi'')) \xrightarrow{\partial} H^{1}(\mathbb{P}^{1}, \bigoplus_{i=1}^{2}\mathcal{O}(\xi_{i}')) \\ 
    \nonumber
    &&\xrightarrow{\delta} H^{1}(\mathbb{P}^{1}, \mathcal{V}_{\text{Log}(\rho)}) \rightarrow H^{1}(\mathbb{P}^{1}, \mathcal{O}(\xi'')) \cdots
    \end{eqnarray}

    For $(2)$, $H^{0}(\mathbb{P}^{1}, \mathcal{O}(-2)^{\oplus 2})$ and $H^{1}(\mathbb{P}^{1}, \mathcal{O}(0))$ are zero-dimensional complex vector spaces while $H^{0}(\mathbb{P}^{1}, \mathcal{O}(0))$ is one-dimensional and $H^{1}(\mathbb{P}^{1}, \mathcal{O}(-2)^{\oplus 2})$ is two-dimensional.  
    This means that $\partial$ is either the $0$ map or injective. If $\partial = 0$, then $\text{dim }H^{0}(\mathbb{P}^{1}, \mathcal{V}_{\text{Log}(\rho)}) = 1$, forcing one of the roots to be $0$ with the other two being negative integers. 
    Moreover, the $\delta$ map has to be an isomorphism, implying that $\text{dim }H^{1}(\mathbb{P}^{1}, \mathcal{V}_{\text{Log}(\rho)}) = 2$. 
    Hence, either $\mathcal{V}_{Log(\rho)} \cong \mathcal{O}(-2) \oplus \mathcal{O}(-2) \oplus \mathcal{O}(0)$ or $\mathcal{V}_{\text{Log}(\rho)} \cong \mathcal{O}(-3) \oplus \mathcal{O}(-1) \oplus \mathcal{O}(0)$. However, the latter case cannot occur by \Cref{NonRoots}. On the other hand, if $\partial$ is injective, then $\text{dim }H^{0}(\mathbb{P}^{1}, \mathcal{V}_{\text{Log}(\rho)}) = 0$, forcing all the roots to be negative. 
    Additionally, $\delta$ has to be surjective, since $\text{dim }H^{1}(\mathbb{P}^{1}, \mathcal{O}(0)) = 0$, but this coerces $\text{dim }H^{1}(\mathbb{P}^{1}, \mathcal{V}_{\text{Log}(\rho)}) = 1$. Thus, $\mathcal{V}_{\text{Log}(\rho)} \cong \mathcal{O}(-2) \oplus \mathcal{O}(-1) \oplus \mathcal{O}(-1)$.

    For $(3)$, $H^{0}(\mathbb{P}^{1}, \mathcal{O}(-2) \oplus \mathcal{O}(-1))$ and $H^{1}(\mathbb{P}^{1}, \mathcal{O}(0))$ are zero-dimensional complex vector spaces while $H^{0}(\mathbb{P}^{1}, \mathcal{O}(0))$ and $H^{1}(\mathbb{P}^{1}, \mathcal{O}(-2) \oplus \mathcal{O}(-1))$ are one-dimensional. 
    It follows that $\partial$ is either the $0$ map or an isomorphism. If $\partial = 0$, then $\text{dim }H^{0}(\mathbb{P}^{1}, \mathcal{V}_{\text{Log}(\rho)}) = 1$, forcing one of the roots to be $0$ with the other two being negative integers. Furthermore, the $\delta$ map has to be an isomorphism, implying that $\text{dim }H^{1}(\mathbb{P}^{1}, \mathcal{V}_{\text{Log}(\rho)}) = 1$.
    Thus, $\mathcal{V}_{\text{Log}(\rho)} \cong \mathcal{O}(-2) \oplus \mathcal{O}(-1) \oplus \mathcal{O}(0)$.
    Now, if $\partial$ is an isomorphism, then $\text{dim }H^{0}(\mathbb{P}^{1}, \mathcal{V}_{\text{Log}(\rho)}) = 0$ and $\text{dim }H^{1}(\mathbb{P}^{1}, \mathcal{V}_{\text{Log}(\rho)}) = 0$, coercing $\mathcal{V}_{\text{Log}(\rho)} \cong \mathcal{O}(-1)^{\oplus 3}$. 

    For $(4)$, we have that $H^{0}(\mathbb{P}^{1}, \mathcal{O}(-2) \oplus \mathcal{O}(0))$, $H^{0}(\mathbb{P}^{1}, \mathcal{O}(0))$, and $H^{1}(\mathbb{P}^{1}, \mathcal{O}(-2) \oplus \mathcal{O}(0))$ are one-dimensional complex vector spaces while $\text{dim }H^{1}(\mathbb{P}^{1}, \mathcal{O}(0))$ = $0$.
    It follows that $\partial$ is either the $0$ map or an isomorphism. If $\partial = 0$, then $H^{0}(\mathbb{P}^{1},\mathcal{V}_{\text{Log}(\rho)}) \cong \mathbb{C}^{2}$, forcing two of the roots to be $0$ while the remaining one is a negative integer. We must also have that $\delta$ is an isomorphism, which shows $H^{1}(\mathbb{P}^{1},\mathcal{V}_{\text{Log}(\rho)})$ is one-dimensional, and hence, $\mathcal{V}_{\text{Log}(\rho)} \cong \mathcal{O}(-2) \oplus \mathcal{O}(0)^{\oplus 2}$.
    On the other hand, if $\partial$ is an isomorphism, then $H^{0}(\mathbb{P}^{1},\mathcal{V}_{\text{Log}(\rho)}) \cong \mathbb{C}$, which implies that one of the roots is $0$ while the remaining two are negative. In addition, we know that $\delta$ is the $0$ map, so $H^{1}(\mathbb{P}^{1},\mathcal{V}_{\text{Log}(\rho)})$ is zero-dimensional. Thus, $\mathcal{V}_{\text{Log}(\rho)} \cong \mathcal{O}(-1)^{\oplus 2} \oplus \mathcal{O}(0)$.
\end{proof}

\begin{Thm} \label{Reducible2}
    Suppose that $0 \rightarrow \chi' \rightarrow \rho \rightarrow \rho'' \rightarrow 0$ is a short exact sequence of monodromy representations of $\mathbb{P}^{1}_{(0, 1, \infty)}$ with $\rho$ three-dimensional, $\chi'$ one-dimensional, and  $\rho''$ two-dimensional. 
    With the principal branch as a choice of logarithm, let $0 \rightarrow \mathcal{O}(\xi') \rightarrow \mathcal{V}_{\emph{Log}(\rho)} \rightarrow \bigoplus_{i = 1}^{2}\mathcal{O}(\xi_{i}'') \rightarrow 0$ be the short exact sequence of the associated extended bundles. Arrange the roots of $\mathcal{V}_{\emph{Log}(\rho'')}$ so that $\xi_{1}'' \leq \xi_{2}''$. Then we have the following set of statements.
    \begin{enumerate}
        \item If $(\xi', \xi_{1}'', \xi_{2}'') \not= (-2,-2,0),(-2,-1,0)$, nor $(-2,0,0)$, then the short exact sequence of associated extended vector bundles $0 \rightarrow \mathcal{O}(\xi') \rightarrow \mathcal{V}_{\emph{Log}(\rho)} \rightarrow \bigoplus_{i = 1}^{2}\mathcal{O}(\xi_{i}'') \rightarrow 0$ splits.
        \item If $(\xi', \xi_{1}'', \xi_{2}'') = (-2,-2,0)$, then the roots of $\mathcal{V}_{\emph{Log}(\rho)}$ are either $(-2,-1,-1), \text{or } (-2,-2,0)$.
        \item If $(\xi', \xi_{1}'', \xi_{2}'') = (-2,-1,0)$, then the roots of $\mathcal{V}_{\emph{Log}(\rho)}$ are either $(-2, -1, 0)$ or $(-1, -1, -1)$.
        \item If $(\xi', \xi_{1}'', \xi_{2}'') = (-2,0,0)$, then the roots of $\mathcal{V}_{\emph{Log}(\rho)}$ are either $(-2,0,0)$ or $(-1,-1,0)$.
    \end{enumerate}
\end{Thm}

\begin{proof}
    The proof is nearly identical to the proof of \Cref{Reducible1}. We explicitly state it below with the appropriate modifications for thoroughness. 
    Using cohomology,
    \begin{eqnarray}
    \nonumber H^{1}(\mathbb{P}^{1}, (\bigoplus\limits_{i = 1}^{2}\mathcal{O}(\xi_{i}'')^{*} \otimes \mathcal{O}(\xi')) && = H^{1}(\bigoplus\limits_{i = 1}^{2}\mathcal{O}(-\xi_{i}'') \otimes \mathcal{O}(\xi')) \\
    \nonumber &&= H^{1}(\mathbb{P}^{1}, \bigoplus_{i=1}^{2}\mathcal{O}(-\xi_{i}''+\xi')).  
    \end{eqnarray}

    \noindent Further, by Serre-duality we have
    \begin{eqnarray}
    \nonumber H^{1}(\mathbb{P}^{1}, \bigoplus_{i=1}^{2}\mathcal{O}(-\xi_{i}''+\xi')) &&= H^{0}(\mathbb{P}^{1}, (\bigoplus_{i=1}^{2}\mathcal{O}(-\xi_{i}''+\xi')^{*} \otimes \Omega_{\mathbb{P}^{1}}^{1}))^{*} \\
    \nonumber &&= H^{0}(\mathbb{P}^{1}, \bigoplus_{i=1}^{2}\mathcal{O}(\xi''-\xi_{i}'-2))^{*}.
    \end{eqnarray}

    \noindent Hence, if 
    \begin{eqnarray} \label{Inequalities2}
        \xi_{1}'' - \xi' < 2 \text{ and } \xi_{2}'' - \xi' < 2,
    \end{eqnarray}
    then $0 \rightarrow \mathcal{O}(\xi') \rightarrow \mathcal{V}_{\text{Log}(\rho)} \rightarrow \bigoplus_{i = 1}^{2}\mathcal{O}(\xi_{i}'') \rightarrow 0$ splits.

    For $(1)$, by \cite[Proposition $4.1$]{DY25}, we know that $0 \geq \xi' \geq -2$, and, moreover, by \Cref{Root-Bound}, we have $0 \geq \xi_{i}'' \geq -2$ for $i=1,2$. Consequently, the inequalities (\ref{Inequalities2}) are not satisfied only when $\xi' = -2$ and at least one of the $\xi_{i}'' = 0$. However, without loss of generality, if $\xi_{2}'' = 0$, then as a consequence of \cite[Theorem $5.9$]{DY25}, $\xi_{1}' = -1$ or $0$ for the case when $\rho'$ is irreducible; whereas when $\rho'$ is reducible, as consequence of \cite[Theorem $4.4$]{DY25}, we also see that $\xi_{1}'$ can be $-2$. This establishes $(1)$. 

    Regarding the remaining three cases, we resort to studying the long exact sequence in cohomology. We have 
    \begin{eqnarray}
    \nonumber
    && 0 \rightarrow H^{0}(\mathbb{P}^{1}, \mathcal{O}(\xi')) \rightarrow H^{0}(\mathbb{P}^{1}, \mathcal{V}_{\text{Log}(\rho)}) \rightarrow H^{0}(\mathbb{P}^{1}, \bigoplus_{i=1}^{2}\mathcal{O}(\xi_{i}'')) \xrightarrow{\partial} H^{1}(\mathbb{P}^{1}, \mathcal{O}(\xi')) \\ 
    \nonumber
    &&\xrightarrow{\delta} H^{1}(\mathbb{P}^{1}, \mathcal{V}_{\text{Log}(\rho)}) \rightarrow H^{1}(\mathbb{P}^{1}, \bigoplus_{i=1}^{2}\mathcal{O}(\xi_{i}'')) \cdots
    \end{eqnarray}

    For $(2)$, $H^{0}(\mathbb{P}^{1}, \mathcal{O}(-2) \oplus \mathcal{O}(0)), H^{1}(\mathbb{P}^{1}, \mathcal{O}(-2))$ and $H^{1}(\mathbb{P}^{1}, \mathcal{O}(-2) \oplus \mathcal{O}(0))$ are one-dimensional complex vector spaces while $H^{0}(\mathbb{P}^{1}, \mathcal{O}(-2))$ is zero-dimensional.  
    This means that $\partial$ is either the $0$ map or an isomorphism. If $\partial = 0$, then $\text{dim }H^{0}(\mathbb{P}^{1}, \mathcal{V}_{\text{Log}(\rho)}) = 1$, forcing one of the roots to be $0$ with the other two being negative integers. 
    Moreover, the $\delta$ map has to be injective, implying that $\text{dim }H^{1}(\mathbb{P}^{1}, \mathcal{V}_{\text{Log}(\rho)}) = 2$. 
    Hence, either $\mathcal{V}_{Log(\rho)} \cong \mathcal{O}(-2) \oplus \mathcal{O}(-2) \oplus \mathcal{O}(0)$ or $\mathcal{V}_{\text{Log}(\rho)} \cong \mathcal{O}(-3) \oplus \mathcal{O}(-1) \oplus \mathcal{O}(0)$. However, the latter case cannot occur by \Cref{NonRoots}. On the other hand, if $\partial$ is an isomorphism, then $\text{dim }H^{0}(\mathbb{P}^{1}, \mathcal{V}_{\text{Log}(\rho)}) = 0$, forcing all the roots to be negative. 
    Additionally, $\text{dim }H^{1}(\mathbb{P}^{1}, \mathcal{O}(-2) \oplus \mathcal{O}(0)) = 1$ coerces $\text{dim }H^{1}(\mathbb{P}^{1}, \mathcal{V}_{\text{Log}(\rho)}) = 1$. Thus, $\mathcal{V}_{\text{Log}(\rho)} \cong \mathcal{O}(-2) \oplus \mathcal{O}(-1) \oplus \mathcal{O}(-1)$.

    For $(3)$, $H^{0}(\mathbb{P}^{1}, \mathcal{O}(-1) \oplus \mathcal{O}(0))$ and $H^{1}(\mathbb{P}^{1}, \mathcal{O}(-2))$ are one-dimensional complex vector spaces while $H^{0}(\mathbb{P}^{1}, \mathcal{O}(-2))$ and $H^{1}(\mathbb{P}^{1}, \mathcal{O}(-1) \oplus \mathcal{O}(0))$ are zero-dimensional. 
    It follows that $\partial$ is either the $0$ map or an isomorphism. If $\partial = 0$, then $\text{dim }H^{0}(\mathbb{P}^{1}, \mathcal{V}_{\text{Log}(\rho)}) = 1$, forcing one of the roots to be $0$ with the other two being negative integers. Furthermore, the $\delta$ map has to be an isomorphism, implying that $\text{dim }H^{1}(\mathbb{P}^{1}, \mathcal{V}_{\text{Log}(\rho)}) = 1$.
    Thus, $\mathcal{V}_{\text{Log}(\rho)} \cong \mathcal{O}(-2) \oplus \mathcal{O}(-1) \oplus \mathcal{O}(0)$.
    Now, if $\partial$ is an isomorphism, then $\text{dim }H^{0}(\mathbb{P}^{1}, \mathcal{V}_{\text{Log}(\rho)}) = 0$ and $\text{dim }H^{1}(\mathbb{P}^{1}, \mathcal{V}_{\text{Log}(\rho)}) = 0$, coercing $\mathcal{V}_{\text{Log}(\rho)} \cong \mathcal{O}(-1)^{\oplus 3}$.

    For $(4)$, we have that $H^{0}(\mathbb{P}^{1}, \mathcal{O}(-2)), H^{1}(\mathbb{P}^{1}, \mathcal{O}(-2))$ are zero-dimensional and one-dimensional complex vector spaces, respectively.
    Additionally, we know that $H^{0}(\mathbb{P}^{1}, \mathcal{O}(0) \oplus \mathcal{O}(0))$ is a two-dimensional complex vector space. 
    It follows that $\partial$ is either the $0$ map or surjective. If $\partial = 0$, then $H^{0}(\mathbb{P}^{1},\mathcal{V}_{\text{Log}(\rho)}) \cong \mathbb{C}^{2}$, forcing two of the roots to be $0$ while the remaining one is a negative integer. We must also have that $\delta$ is an isomorphism, which shows $H^{1}(\mathbb{P}^{1},\mathcal{V}_{\text{Log}(\rho)})$ is one-dimensional, and hence, $\mathcal{V}_{\text{Log}(\rho)} \cong \mathcal{O}(-2) \oplus \mathcal{O}(0)^{\oplus 2}$.
    On the other hand, if $\partial$ is surjective, then $H^{0}(\mathbb{P}^{1},\mathcal{V}_{\text{Log}(\rho)}) \cong \mathbb{C}$, which implies that one of the roots is $0$ while the remaining two are negative. In addition, we know that $\delta$ is the $0$ map, so $H^{1}(\mathbb{P}^{1},\mathcal{V}_{\text{Log}(\rho)})$ is zero-dimensional. Thus, $\mathcal{V}_{\text{Log}(\rho)} \cong \mathcal{O}(-1)^{\oplus 2} \oplus \mathcal{O}(0)$.
\end{proof}

\begin{Rem}
    To conclude this section, we state a direct corollary from the previous two theorems that will give further credence to the statements posed in \Cref{Expectations}. 
\end{Rem}

\begin{Cor}\label{ReducibleBound}
    Let $\rho: \mathbb{Z}*\mathbb{Z} \rightarrow \text{GL}_{3}(\mathbb{C})$ be a reducible monodromy representation of $\mathbb{P}^{1}_{(0, 1, \infty)}$. Then the roots of $\mathcal{V}_{\text{Log}(\rho)} \cong \bigoplus\limits_{i=1}^{3}\mathcal{O}(\xi_{i})$ are bounded: 
    \begin{eqnarray}
        \nonumber
        0 \geq \xi_{i} >-3 
    \end{eqnarray}
    for each $1 \leq i \leq 3$. 
\end{Cor}

\begin{proof}
    Combine \Cref{Reducible1} with \Cref{Reducible2}.
\end{proof}

\section{Example from the Projective Modular Group}
In \cite[Section $6$]{DY25} we present an example of computing the roots of an irreducible monodromy representation. Here we will complement that example by taking a reducible but indecomposable representation. 

Classification of such representations for $\text{PSL}_{2}(\mathbb{Z})$ have been established in \cite{CHMY2018}; we take a reducible but indecomposable three-dimensional monodromy representation $\rho$ of $\mathbb{P}^{1}_{(0,1,\infty)}$ that factors through $\text{PSL}_{2}(\mathbb{Z})$ from \cite[Table $1$]{CHMY2018}. 

We generate the modular group $\text{SL}_{2}(\mathbb{Z})$ by 
\begin{eqnarray}
    S = 
    \begin{pmatrix}
        0 && -1 \\
        1 && 0 
    \end{pmatrix}, \text{  }
    T = 
    \begin{pmatrix}
        1 && 1 \\
        0 && 1 
    \end{pmatrix},
\end{eqnarray}
subject to the relations $S^{2} = (ST)^{3}$ and $S^{4} = 1$.
We will continue to use the notation $S$ and $T$ to refer to the images of these matrices in the quotient group $\text{PSL}_{2}(\mathbb{Z}) = \text{SL}_{2}{\mathbb{Z}} / \{\pm 1\}$. 
Therefore, a function $\varrho: \text{PSL}_{2}(\mathbb{Z}) \rightarrow \text{GL}_{3}(\mathbb{C})$ defines a matrix representation of $\text{PSL}_{2}(\mathbb{Z})$ if and only if $\varrho(S)^{2} = \varrho(ST)^{3} = 1$. Additionally, following \cite{CHMY2018}, we assume that $\varrho(T)$ is diagonalizable, and that a basis has been chosen so that
\begin{eqnarray}
    \varrho(T) = \text{diag}\{\lambda_{1}, \lambda_{2}, \lambda_{3}\}.
\end{eqnarray}
Define a map $\phi: \mathbb{Z}*\mathbb{Z} \rightarrow \text{PSL}_{2}(\mathbb{Z})$ by $\gamma_{0} \mapsto S$ and $\gamma_{1} \mapsto T$, and compose it with $\varrho$ so that 
\begin{center}
    \begin{tikzcd}
        \mathbb{Z}*\mathbb{Z} \arrow[rd, "\rho"] \arrow[r, "\phi"] & \text{PSL}_{2}(\mathbb{Z}) \arrow[d, "\varrho"] \\
        & \text{GL}_{3}(\mathbb{C})
    \end{tikzcd}
\end{center}
the composition $\rho$ is a monodromy representation of $\mathbb{P}^{1}_{(0,1,\infty)}$. 

It follows from \cite[Table $1$]{CHMY2018} that 
\begin{eqnarray}\label{ReducibleExample}
    \rho(\gamma_{0}) =
    \begin{pmatrix}
        1 & 0 & 0 \\
        0 & e^{2\pi i (\frac{5}{6})} & 0 \\
        0 & 0 & e^{2\pi i (\frac{1}{6})}
    \end{pmatrix} \text{ and }
    \rho(\gamma_{1}) =
    \begin{pmatrix}
        1 & 1 & 1 \\
        0 & -1 & 0 \\
        0 & 0 & -1
    \end{pmatrix}
\end{eqnarray}
defines a reducible but indecomposable three-dimensional monodromy representation of $\mathbb{P}^{1}_{(0,1,\infty)}$ that factors through $\text{PSL}_{2}(\mathbb{Z})$. 
Now let $\rho'$ be the subrepresentation defined by the upper $2 \times 2$ blocks. Since $\rho'$ is reducible, we have the short exact sequence 
\begin{eqnarray}
    0 \rightarrow \chi_{1} \rightarrow \rho' \rightarrow \chi_{2} \rightarrow 0
\end{eqnarray}
with $\chi_{1}, \chi_{2}$ characters representing the order pairs $(1,1), (e^{2\pi i (\frac{5}{6})}, -1)$ respectively. 

From \cite[Proposition $4.1$]{DY25}, we know that $\mathcal{V}_{\text{Log}(\chi_{1})} \cong \mathcal{O}(0)$ and $\mathcal{V}_{\text{Log}(\chi_{2})} \cong \mathcal{O}(-2)$. At this point, we apply \cite[Theorem $4.4(1)$]{DY25} to see that $\mathcal{V}_{\text{Log}(\rho')} \cong \mathcal{O}(-2) \oplus \mathcal{O}(0)$. From Equation (\ref{ReducibleExample}) we started with the short exact sequence
\begin{eqnarray}
    0 \rightarrow \rho' \rightarrow \rho \rightarrow \chi'' \rightarrow 0
\end{eqnarray}
where we now know that the corresponding short exact sequence of associated extended bundles is 
\begin{eqnarray}
    0 \rightarrow \mathcal{O}(-2) \oplus \mathcal{O}(0) \rightarrow \mathcal{V}_{\text{Log}(\rho)} \rightarrow \mathcal{V}_{\text{Log}(\chi'')} \rightarrow 0 
\end{eqnarray}

In order to apply \Cref{Reducible1}, we need to know the root of $\mathcal{V}_{\text{Log}(\chi'')}$. However, from \cite[Proposition $4.1$]{DY25} we see that $\mathcal{V}_{\text{Log}(\chi'')} \cong \mathcal{O}(-1)$. Thus, as a consequence of \Cref{Reducible1}$(1)$, the short exact sequence 
\begin{eqnarray}
    0 \rightarrow \mathcal{O}(-2) \oplus \mathcal{O}(0) \rightarrow \mathcal{V}_{\text{Log}(\rho)} \rightarrow \mathcal{O}(-1) \rightarrow 0. 
\end{eqnarray}
splits, showing that the roots of $\rho$ are $(-2,0,-1)$.

\section{Irreducible Monodromy Representations} \label{IrreducibleSection}
The goal of this section is to compute the roots of any vector bundle arising from an irreducible three-dimensional monodromy representation of $\mathbb{P}^{1}_{(0, 1, \infty)}$. The key force in the main theorem of this section is the following lemma which generalizes \cite[Lemma $5.8$]{DY25}. 

\begin{Lm} \label{Simple-Module}
    Let $\rho$ be an $n$-dimensional monodromy representation of $\mathbb{P}^{1}_{(p_{1}, ..., p_{m})}$ with $m,n \geq 2$. If $\rho$ is irreducible, then no proper submodule of $\mathcal{N}(\rho)$ is closed under the monodromy derivative $D$, i.e., $\mathcal{N}(\rho)$ is a simple as a $D$-module. 
\end{Lm}

\begin{proof}
    In search of a contradiction, assume that there is some submodule $L$ of $\mathcal{N}(\rho)$ that is closed under the monodromy derivative $D$. Note that as $\mathbb{C}[x,y]$ is a Noetherian ring and $\mathcal{N}(\rho)$ is finitely generated, $L$ is finitely generated; let $\{f_{1}, ..., f_{k}\}$ be a generating set. 
    Observe that each $f_{i}|_{\mathbb{P}^{1}_{(p_{1}, ..., p_{m})}}$ is a nonzero global section of $\mathcal{Q}_{\rho}$ (see \Cref{Associated-Con}) since $\big(\mathcal{V}_{\text{Log}(\rho)} \otimes \mathcal{O}(\xi)\big)|_{\mathbb{P}^{1}_{(p_{1}, ..., p_{m})}} \cong \mathcal{Q}_{\rho}$. 
    Denote with $\mathcal{E}_{L} \subset \mathcal{Q}_{\rho}$ the subsheaf generated by $L$ on $\mathbb{P}^{1}_{(p_{1}, ..., p_{m})}$.
    Now, denoting with $\varphi$ the restriction maps, as the following diagram 
    \begin{center}
    \begin{tikzcd}
    \mathcal{V}_{\text{Log}(\rho)} \otimes \mathcal{O}(\xi) \arrow[r, "\nabla_{\text{Log}(\rho)} \otimes \nabla"] \arrow[d, "\varphi"]
    & \mathcal{V}_{\text{Log}(\rho)} \otimes \mathcal{O}(\xi+1) \arrow[d, "\varphi"] \\
    \mathcal{Q}_{\rho} \arrow[r, "(\nabla_{\text{Log}(\rho)} \otimes \nabla)|_{\mathbb{P}^{1}_{(p_{1}, ..., p_{m})}}"]
    & \mathcal{Q}_{\rho}  
    \end{tikzcd}
    \end{center}
    commutes, it follows that $(Df_{i})|_{\mathbb{P}^{1}_{(p_{1}, ..., p_{m})}} = D|_{\mathbb{P}^{1}_{(p_{1}, ..., p_{m})}}(f_{i}|_{\mathbb{P}^{1}_{(p_{1}, ..., p_{m})}})$. 
    As every bundle over the affine space $\mathbb{P}^{1}_{(p_{1}, ..., p_{m})}= \text{Spec}(A)$ is free by \Cref{Free-sheaf}, $\mathcal{E}_{L}$ can be viewed as a free module of rank $k$ over $A$ whose generators are stable under $D|_{\mathbb{P}^{1}_{(p_{1}, ..., p_{m})}} \cong \nabla_{\rho}$, then it is clear that $\mathcal{E}_{L}$ inherits the connection map $\nabla_{\rho}|_{\mathcal{E}_{L}}$ in such a way where $(\mathcal{E}_{L},\nabla_{\rho}|_{\mathcal{E}_{L}}) \subset (\mathcal{Q}_{\rho}, \nabla_{\rho})$ as connections.
    At this point observe that since there is an equivalence of categories between holomorphic connections on $\mathbb{P}^{1}_{(p_{1}, ..., p_{m})}$ and monodromy representations of $\mathbb{P}^{1}_{(p_{1}, ..., p_{m})}$, then $\rho' \subset \rho$, where $\rho'$ corresponds to $(\mathcal{E}_{L}, \nabla_{\rho}|_{\mathcal{E}_{L}})$. 
    Examining $(\mathcal{E}_{L}, \nabla_{\rho}|_{\mathcal{E}_{L}})$, one sees that $(\mathcal{E}_{L}, \nabla_{\rho}|_{\mathcal{E}_{L}})$ is nonempty, which means that $\rho' \not= 0$; furthermore, as $\mathcal{N}(\rho)$ is of rank $n$, it follows that $(\mathcal{E}_{L}, \nabla_{\rho}|_{\mathcal{E}_{L}}) \subsetneqq (\mathcal{Q}_{\rho}, \nabla_{\rho})$, which forces $\rho' \subsetneqq \rho$. 
    Consequently, $\rho'$ must be a proper subrepresentation of $\rho$; however this contradicts the initial assumption that $\rho$ was irreducible. 
\end{proof}

\begin{Thm}\label{IrreducibleReps}
    Suppose that $\rho: \mathbb{Z}*\mathbb{Z} \rightarrow \text{GL}_{3}(\mathbb{C})$ is an irreducible three-dimensional monodromy representation of $\mathbb{P}^{1}_{(0,1,\infty)}$ and let $(\mathcal{V}_{\text{Log}(\rho)}, \nabla_{\text{Log}(\rho)})$ denote the associated extended logarithmic connection. 
    Allow $\zeta = c_{1}(\mathcal{V}_{\text{Log}(\rho)})$.
    Then, 
    \begin{eqnarray}
        \nonumber
        \mathcal{V}_{\text{Log}(\rho)} \cong
        \begin{cases}
        \mathcal{O}(\frac{\zeta}{3})^{\oplus 3} \text{or } \mathcal{O}(\frac{\zeta}{3} + 1) \oplus \mathcal{O}(\frac{\zeta}{3}) \oplus \mathcal{O}(\frac{\zeta}{3} - 1) & \text{when } \zeta \equiv 0 \text{ mod } 3\\
        \mathcal{O}(\frac{\zeta + 2}{3}) \oplus \mathcal{O}(\frac{\zeta-1}{3}) \oplus \mathcal{O}(\frac{\zeta-1}{3}) & \text{when } \zeta \equiv 1 \text{ mod } 3 \\
        \mathcal{O}(\frac{\zeta+1}{3}) \oplus \mathcal{O}(\frac{\zeta+1}{3}) \oplus \mathcal{O}(\frac{\zeta-2}{3}) & \text{when } \zeta \equiv 2 \text{ mod } 3.
    \end{cases}
    \end{eqnarray}
\end{Thm}

\begin{proof}
    Since $\rho$ is three-dimensional, $\mathcal{N}(\rho)$ has rank three, so let $F$, $G$, and $H$ be a set of generators. 
    Further, let $\xi_{\text{min}}, \xi_{1}$, and $\xi_{2}$ be the weights of $F,G$, and $H$, respectively, where by assumption $\xi_{\text{min}}$ is the minimal weight of $\mathcal{N}(\rho)$, with $\xi_{\text{min}} \leq \xi_{1}, \xi_{2}$.
    
    By construction, $D_{\xi_{\text{min}}}(F) \in \mathcal{N}_{\xi_{\text{min}}+1}(\rho)$, and so $D_{\xi_{\text{min}}}(F)$ has weight $\xi_{\text{min}} +1$. 
    Moreover, since $F,G$, and $H$ generate $\mathcal{N}(\rho)$, we must have $D_{\xi_{\text{min}}}(F) = \eta_{F} F + \tau_{F} G + \omega_{F} H$ for some $\eta_{F}, \tau_{F}, \omega_{F} \in \mathbb{C}[x,y]$, not all of which are $0$, by \cite[Lemma $5.8$]{DY25}.
    If $\tau_{F} = \omega_{F} = 0$, then $D_{\xi_{\text{min}}}(F) = \eta_{F} F$, which contradicts \Cref{Simple-Module}. So, without loss of generality, suppose that $\tau_{F} \not= 0$. It follows that $\tau_{F} G$ is of weight $\xi_{\text{min}} + 1$, implying that $\xi_{\text{min}} + 1 = \xi_{1} + \mathfrak{t}_{F}$, where $\text{deg}(\tau_{F}) = \mathfrak{t}_{F}$. Ergo, 
    \begin{eqnarray}
    \nonumber
    \xi_{1} = 
    \begin{cases}
        \xi_{\text{min}} + 1 & \text{when } \mathfrak{t}_{F}=0 \\
        \xi_{\text{min}} & \text{when } \mathfrak{t}_{F}=1
    \end{cases}
    \end{eqnarray}
 
    \noindent since $0 \leq \mathfrak{t}_{F}$ and $\xi_{\text{min}} \leq \xi_{1}$.

    Again, by construction, $D_{\xi_{1}}(G) \in \mathcal{N}_{\xi_{1}+1}(\rho)$, and so $D_{\xi_{1}}(G)$ has weight $\xi_{1} +1$. 
    Moreover, since $F,G$, and $H$ generate $\mathcal{N}(\rho)$, we must have $D_{\xi_{1}}(G) = \eta_{G} F + \tau_{G} G + \omega_{G} H$ for some $\eta_{G}, \tau_{G}, \omega_{G} \in \mathbb{C}[x,y]$, not all of which are $0$, by \cite[Lemma $5.8$]{DY25}.
    If $\eta_{G} = \omega_{G} = 0$, then $D_{\xi_{1}}(G) = \tau_{G} G$, which contradicts $\Cref{Simple-Module}$. 
    Witness that if $\omega_{G} = 0$, then $\omega_{F} \not= 0$, as otherwise $F$ and $G$ generate a submodule of $\mathcal{N}(\rho)$ that is closed under the monodromy derivative, which contradicts \Cref{Simple-Module}. Hence, if $\omega_{G} = 0$, then it follows that $\omega_{F} H$ is of weight $\xi_{\text{min}} + 1$, implying that $\xi_{\text{min}} + 1 = \xi_{2} + \mathfrak{w}_{F}$, where $\text{deg}(\omega_{F}) = \mathfrak{w}_{F}$. Ergo, 
    \begin{eqnarray}
    \nonumber
    \xi_{2} = 
    \begin{cases}
        \xi_{\text{min}} + 1 & \text{when } \mathfrak{w}_{F}=0 \\
        \xi_{\text{min}} & \text{when } \mathfrak{w}_{F}=1
    \end{cases}
    \end{eqnarray}
 
    \noindent since $0 \leq \mathfrak{w}_{F}$ and $\xi_{\text{min}} \leq \xi_{2}$. 
    On the other hand, if $\omega_{G} \not= 0$, then $\omega_{G}H$ is of weight $\xi_{1} + 1$, implying that $\xi_{1} + 1 = \xi_{2} + \mathfrak{w}_{G}$, where $\text{deg}(\omega_{G}) = \mathfrak{w}_{G}$. Ergo, when $\xi_{1} = \xi_{\text{min}}$,
    \begin{eqnarray}
    \nonumber
    \xi_{2} = 
    \begin{cases}
        \xi_{\text{min}} + 1 & \text{when } \mathfrak{w}_{G}=0 \\
        \xi_{\text{min}} & \text{when } \mathfrak{w}_{G}=1
    \end{cases}
    \end{eqnarray}

    \noindent since $0 \leq \mathfrak{w}_{G}$ and $\xi_{\text{min}} \leq \xi_{2}$, and when $\xi_{1} = \xi_{\text{min}}+1$,
    \begin{eqnarray}
    \nonumber
    \xi_{2} = 
    \begin{cases}
        \xi_{\text{min}} + 2 & \text{when } \mathfrak{w}_{G}=0 \\
        \xi_{\text{min}} + 1 & \text{when } \mathfrak{w}_{G}=1 \\
        \xi_{\text{min}} & \text{when } \mathfrak{w}_{G}=2
    \end{cases}
    \end{eqnarray}

    \noindent since $0 \leq \mathfrak{w}_{G}$ and $\xi_{\text{min}} \leq \xi_{2}$.

    \begin{figure}
        \centering
        \begin{tikzpicture}
        \node(0){$-\xi_{\text{min}}$}
        child[sibling distance=29mm]{node{$-\xi_{\text{min}}$}
            child{node{$-\xi_{\text{min}}$}}
            child{node{$-\xi_{\text{min}} - 1$}}
        }
        child[sibling distance=29mm]{node{$-\xi_{\text{min}} - 1$}
            child{node{$-\xi_{\text{min}} - 1$}}
            child{node{$-\xi_{\text{min}} - 2$}}
        };
        \end{tikzpicture}
        \caption{Tree of Roots}
        \label{TreeFig}
    \end{figure}

    Now, as the weights of the generators of $\mathcal{N}(\rho)$ are the negatives of the roots of $\mathcal{V}_{\text{Log}(\rho)}$ by \Cref{ModuleProperties}, we have four combinations for the roots, given by \Cref{TreeFig}.
    We read \Cref{TreeFig} from left to right and label each path from one to four.

    Thus, by \Cref{First-Chern-Class}, we have $\zeta = -3\xi_{\text{min}} - \kappa$ where $\kappa \in \{0,1,2,3\}$, and considering $\zeta$ mod $3$ leads to: case $(1)$ and $(4)$, from \Cref{TreeFig}, belonging to $\zeta \equiv 0 \text{ mod } 3$, case $(2)$ belonging to $\zeta \equiv 2 \text{ mod } 3$, and case $(3)$ belonging to $\zeta \equiv 1 \text{ mod } 3$. Lastly, the result follows after one notes that $-\xi_{\text{min}} = \frac{\zeta + \kappa}{3}$.
\end{proof}

\begin{Rem}\label{Expectations}
    We now state the final remarks.
    \begin{enumerate}
        \item Examining the use of the monodromy derivative in the proof of \Cref{IrreducibleReps}, combined with \Cref{Degree-of-D}, one expects that the roots of any irreducible monodromy representation of dimension $d$ of $\mathbb{P}^{1}_{(0,1,\infty)}$ can be given by continuing the tree diagram in \Cref{TreeFig} up to depth $d$. 
        Continuing along the same line of thought, given an irreducible monodromy representation of dimension $d$ of $\mathbb{P}^{1}_{(p_{1}\cdots p_{m})}$, one can expect a hypertree analogous to \Cref{TreeFig} where the minimum weight is directly connected to all integers in the interval $[-\xi_{\text{min}}, -\xi_{\text{min}}-(m-2)]$ when $m \geq 2$, and continued along in a similar fashion as of the proof of \Cref{IrreducibleReps} up to depth $d$.
        \item Following \Cref{Root-Bound}, \Cref{ReducibleBound}, and the bound found in \cite[\S $2.1$]{MT2021}, we might also expect the roots of any $d$-dimensional monodromy representation (reducible or irreducible) $\rho$ of $\mathbb{P}^{1}_{(p_{1}\cdots p_{m})}$ to be bounded by $-m < \xi_{i} \leq 0$ for $ 1 \leq i \leq d$.
        \item The monodromy derivative has allowed us to translate the properties of the associated extended vector bundles arising from a monodromy representation into purely alegrbaic properties of $\mathcal{N}(\rho)$. Particularly, in the case of irreducible monodromy representations, the monodromy derivative along with the first Chern class is sufficient for obtaining the decomposition.   
    \end{enumerate}
\end{Rem}

\bibliographystyle{siam}
\bibliography{References}

\begin{thebibliography}{10}

\bibitem{Lak2003}
{\sc V.~Balaji, V.~Mehta, K.~Nagarajan, K.~Pranjape, P.~Sankaran, and
  R.~Sridharan}, {\em \emph{``A tribute to CS Seshadri: Perspectives in
  geometry and representation Theory"}}, Springer, 2003.

\bibitem{CF2017}
{\sc L.~Candelori and C.~Franc}, {\em Vector-valued modular forms and the
  modular orbifold of elliptic curves}, International Journal of Number Theory,
  13 (2017), pp.~39--63.

\bibitem{CHMY2018}
{\sc L.~Candelori, T.~Hartland, C.~Marks, and D.~Y{\'e}pez}, {\em
  Indecomposable vector-valued modular forms and periods of modular curves},
  Research in Number Theory, \textbf{4} (2018), pp.~1--24.

\bibitem{Deligne2006}
{\sc P.~Deligne}, {\em \emph{``{\'E}quations diff{\'e}rentielles {\`a} points
  singuliers r{\'e}guliers}"}, vol.~163, Springer, 2006.

\bibitem{DP2022}
{\sc R.~Donagi and T.~Pantev}, {\em Parabolic hecke eigensheaves}, Astérisque,
  \textbf{435} (2022).

\bibitem{Forster2012}
{\sc O.~Forster}, {\em \emph{``Lectures on Riemann surfaces"}}, vol.~81,
  Springer Science \& Business Media, 2012.

\bibitem{FR2020}
{\sc C.~Franc and S.~Rayan}, {\em Nonabelian hodge theory and vector valued
  modular forms}, Vertex Operator Algebras, Number Theory and Related Topics,
  (2020), pp.~95--118.

\bibitem{G1957}
{\sc A.~Grothendieck}, {\em Sur la classification des fibr{\'e}s holomorphes
  sur la sphere de riemann}, American Journal of Mathematics, 79 (1957),
  pp.~121--138.

\bibitem{Hu2024}
{\sc Z.~Hu, P.~Huang, and R.~Zong}, {\em Moduli spaces of parabolic bundles
  over $\mathbb{P}^{1}$ with five marked points}, arXiv preprint
  arXiv:2108.08994v3,  (2024).

\bibitem{Inaba2006}
{\sc M.-a. Inaba, K.~Iwasaki, and M.-H. Saito}, {\em Moduli of stable parabolic
  connections, riemann--hilbert correspondence and geometry of painlev{\'e}
  equation of type vi, part i}, Publications of the Research Institute for
  Mathematical Sciences, \textbf{42} (2006), pp.~987--1089.

\bibitem{Inaba2007}
\leavevmode\vrule height 2pt depth -1.6pt width 23pt, {\em Moduli of stable
  parabolic connections, riemann-hilbert correspondence and geometry of
  painlev{\'e} equation of type vi, part ii}, arXiv preprint
  arXiv:math/0605025v2,  (2007).

\bibitem{IKSY2013}
{\sc K.~Iwasaki, H.~Kimura, S.~Shimemura, and M.~Yoshida}, {\em \emph{``From
  Gauss to Painlev{\'e}: a modern theory of special functions"}}, vol.~16,
  Springer Science \& Business Media, 2013.

\bibitem{KLS2022}
{\sc A.~Komyo, F.~Loray, and M.-H. Saito}, {\em Moduli space of irregular rank
  two parabolic bundles over the riemann sphere and its compactification},
  Advances in Mathematics, 410 (2022), p.~108750.

\bibitem{MM2010}
{\sc C.~Marks and G.~Mason}, {\em Structure of the module of vector-valued
  modular forms}, Journal of the London Mathematical Society, 82 (2010),
  pp.~32--48.

\bibitem{Mason2008}
{\sc G.~Mason}, {\em 2-dimensional vector-valued modular forms}, The Ramanujan
  Journal, 17 (2008), pp.~405--427.

\bibitem{Mats2024}
{\sc T.~Matsumoto}, {\em Moduli space of rank three logarithmic connections on
  the projective line with three poles}, arXiv preprint arXiv:2311.10071v4,
  (2024).

\bibitem{MS1980}
{\sc V.~B. Mehta and C.~S. Seshadri}, {\em Moduli of vector bundles on curves
  with parabolic structures}, Mathematische annalen, \textbf{248} (1980),
  pp.~205--239.

\bibitem{MT2021}
{\sc C.~Meneses and L.~A. Takhtajan}, {\em Logarithmic connections, wznw
  action, and moduli of parabolic bundles on the sphere}, Communications in
  Mathematical Physics, \textbf{387} (2021), pp.~649--680.

\bibitem{NS1965}
{\sc M.~S. Narasimhan and C.~S. Seshadri}, {\em Stable and unitary vector
  bundles on a compact riemann surface}, Annals of Mathematics, 82 (1965),
  pp.~540--567.

\bibitem{Ohtsuki1982}
{\sc M.~Ohtsuki}, {\em A residue formula for chern classes associated with
  logarithmic connections}, Tokyo Journal of Mathematics, \textbf{5} (1982),
  pp.~13--21.

\bibitem{Szamuely2009}
{\sc T.~Szamuely}, {\em \emph{``Galois groups and fundamental groups"}},
  vol.~117, Cambridge university press, 2009.

\bibitem{TZ2008}
{\sc L.~A. Takhtajan and P.~Zograf}, {\em The first chern form on moduli of
  parabolic bundles}, Mathematische Annalen, 341 (2008), pp.~113--135.

\bibitem{DY25}
{\sc D.~Y{\'e}pez}, {\em Computing the roots of twisting sheaves over the
  projective line arising from monodromy representations}, arXiv preprint
  arXiv:2501.01006,  (2025).

\end{thebibliography}

\end{document}